\documentclass[12pt]{article}
\usepackage{latexsym,amssymb,amsmath,amsthm,enumerate,float,geometry,cite}
\newtheorem{theorem}{Theorem}

\newtheorem{fact}[theorem]{Fact}

\newtheorem{claim}{Claim}

\usepackage{lineno}
\usepackage{setspace}

\begin{document}
\onehalfspace

\title{Exponential Independence in Subcubic Graphs}
\author{St\'{e}phane Bessy$^1$ 
\and Johannes Pardey$^2$ 
\and Dieter Rautenbach$^2$}
\date{}

\maketitle
\vspace{-10mm}
\begin{center}
{\small 
$^1$ Laboratoire d'Informatique, de Robotique et de Micro\'{e}lectronique de Montpellier,\\ 
Montpellier, France, \texttt{stephane.bessy@lirmm.fr}\\[3mm] 
$^2$ Institute of Optimization and Operations Research, Ulm University,\\ 
Ulm, Germany, \texttt{$\{$johannes.pardey,dieter.rautenbach$\}$@uni-ulm.de}}
\end{center}

\begin{abstract}
A set $S$ of vertices of a graph $G$ is {\it exponentially independent} if, for every vertex $u$ in $S$,
$$\sum\limits_{v\in S\setminus \{ u\}}\left(\frac{1}{2}\right)^{{\rm dist}_{(G,S)}(u,v)-1}<1,$$
where ${\rm dist}_{(G,S)}(u,v)$ is the distance between $u$ and $v$ in the graph $G-(S\setminus \{ u,v\})$.
The {\it exponential independence number} $\alpha_e(G)$ of $G$
is the maximum order of an exponentially independent set in $G$.
In the present paper we present several bounds on this parameter and 
highlight some of the many related open problems.
In particular, we prove that 
subcubic graphs of order $n$ 
have exponentially independent sets 
of order $\Omega(n/\log^2(n))$,
that the infinite cubic tree has no exponentially independent set of positive density,
and that subcubic trees of order $n$ 
have exponentially independent sets 
of order $(n+3)/4$.\\[3mm]
\end{abstract}

\section{Introduction}

Independent and dominating sets in graphs 
are among 
the most natural and well studied concepts in graph theory.
Inspired by the notion of exponential domination introduced by Dankelmann et al.~\cite{ddems},
J\"{a}ger et al.~\cite{jr} introduced exponential independence,
where the influence of vertices on each other decays exponentially with the pairwise distance,
and vertices may block each others influence.
More precisely, for a set $S$ of vertices of a graph $G$, and a vertex $u$ of $G$, let  
\begin{eqnarray*}
w_{(G,S)}(u)=\sum\limits_{v\in S}\left(\frac{1}{2}\right)^{{\rm dist}_{(G,S)}(u,v)-1},
\end{eqnarray*}
where ${\rm dist}_{(G,S)}(u,v)$ is the distance between $u$ and $v$ in the graph $G-(S\setminus \{ u,v\})$.
Now, the set $S$ is {\it exponentially independent in $G$} if
\begin{eqnarray*}
w_{(G,S\setminus \{ u\})}(u)<1\mbox{ for every vertex $u$ in $S$},
\end{eqnarray*}
and the {\it exponential independence number} $\alpha_e(G)$ of $G$
is the maximum order of an exponentially independent set in $G$.
Dankelmann et al.~\cite{ddems} define $S$ to be {\it exponentially dominating in $G$} if
$w_{(G,S)}(u)\geq 1$ for every vertex $u$ of $G$, 
and the {\it exponential domination number $\gamma_e(G)$} of $G$ 
as the minimum order of an exponentially dominating set in $G$.

Exponential domination and exponential independence have been studied in a number of papers \cite{abcvy,aa1,aa2,bor1,bor2,ca,hjr1,hjr2}
but many seemingly simple fundamental problems remain open.
Even for trees the largest possible value of the exponential domination number 
as well as its computational complexity are unknown.

In the present paper we present several results on the exponential independence number and 
highlight many open problems related to it.
Our main results are that
subcubic graphs of order $n$ 
have exponentially independent sets 
of order $\Omega(n/\log^2(n))$,
that the infinite cubic tree has no exponentially independent set of positive density,
and that subcubic trees of order $n$ 
have exponentially independent sets 
of order $(n+3)/4$.
We consider simple and undirected graphs, and use standard terminology.
Unless we explicitly say otherwise, all graphs are finite.
All logarithms have base $2$.

\section{Results}

For a connected subcubic graph $G$ of order $n$, it is known \cite{jr} that 
\begin{eqnarray}
\Omega(\log(n))\leq \alpha_e(G)\leq \frac{n+1}{2}.\label{e0}
\end{eqnarray}
While the upper bound in (\ref{e0}) is satisfied with equality for full binary trees,
J\"{a}ger et al.~\cite{jr} conjectured that the lower bound can be improved considerably.
Our first result confirms this.

\begin{theorem}\label{theorem1}
If $G$ is a subcubic graph of order $n$, then 
$$\alpha_e(G)\geq \frac{n}{3\cdot 2^6\cdot \log^2(n)}.$$
\end{theorem}
\begin{proof} Let $G$ be a subcubic graph of order $n$.
Clearly, we may assume that $n>3\cdot 2^6$.
Let 
$$d^*=\left\lceil\log(\log(n))\right\rceil+2\leq \log(\log(n))+3.$$
Let $S$ be a maximal set of vertices of $G$ such that 
the distance ${\rm dist}_G(u,v)$ in $G$ between any two distinct vertices $u$ and $v$ from $S$ 
is more than $2d^*$.
Since there are at most
$$1+3\cdot 2^0+3\cdot 2^1+\cdots+3\cdot 2^{2d^*-1}
=3\cdot 2^{2d^*}-2
\leq 3\cdot 2^{2\log(\log(n))+6}
=3\cdot 2^6\cdot \log^2(n)$$
vertices at distance at most $2d^*$ from any vertex of $G$,
the set $S$ has order at least $$\frac{n}{3\cdot 2^6\cdot \log^2(n)}.$$
In order to complete the proof, we show that $S$ is exponentially independent.
Therefore, let $u$ be any vertex in $S$.
For every vertex $v$ in $S\setminus \{ u\}$ 
with ${\rm dist}_{(G,S\setminus\{ u\})}(u,v)<\infty$,
fix a shortest path $P_v$ in $G-(S\setminus \{ u,v\})$ between $u$ and $v$.
Let these paths $P_v$ be chosen in such a way that their union $T$ has as few vertices and edges as possible.
It follows easily that $T$ is a subcubic tree.
We select $u$ as the root of $T$.
By construction, the leaves of $T$ are exactly the vertices in $S\setminus \{ u\}$ with ${\rm dist}_{(G,S\setminus\{ u\})}(u,v)<\infty$.
For every positive integer $i$,
let $n_i$ be the number of vertices $x$ of $T$ of depth ${\rm dist}_T(u,x)=i$,
and let $\alpha_i$ be the number of leaves of $T$ of depth $i$.
Since 
$$w_{(G,S\setminus \{ u\})}(u)=\sum\limits_{i=1}^{\infty}2^{-i+1}\alpha_i,$$
it remains to bound this quantity.
Since $T$ is a subcubic subtree of $G$ rooted in $u$, we have
\begin{eqnarray}
n_1 & \leq & 3\label{e1},\\
n_i & \leq & 2n_{i-1}\mbox{ for $i\geq 2$, and}\label{e2}\\
\sum\limits_{i=1}^{\infty}n_i & \leq & n-1\label{e3}.
\end{eqnarray}
By the choice of $S$, we have $\alpha_i=0$ for $i\leq 2d^*$, 
and $\alpha_i\leq n_i$ for every $i$.
Furthermore, if $v$ is a leaf of $T$ of depth $i$, that is, in particular, $\alpha_i>0$, and, hence, $i\geq 2d^*$,
and $p_v$ is the ancestor of $v$ of depth $i-d^*$, then the choice of $S$ implies that $v$ is the only 
descendant of $p_v$ of depth $i$ that is a leaf of $T$, that is, the function $v\mapsto p_v$ is injective.
This implies that 
$$\alpha_i\leq n_{i-d^*},$$ 
and, hence $w_{(G,S\setminus \{ u\})}(u)$ is at most
\begin{eqnarray}
\sum\limits_{i=d^*+1}^{\infty}2^{-i+1}n_{i-d^*}\label{e4}.
\end{eqnarray}
Now, choose the non-negative integers $n_1,n_2,\ldots$ 
(independently of $G$ and $T$)
such that (\ref{e1}), (\ref{e2}), and (\ref{e3}) are satisfied,
and (\ref{e4}) is as large as possible subject to these conditions.
Let $D$ be the largest integer with $n_D>0$.
Suppose, for a contradiction, that $D>\log\left(\frac{n+1}{3}\right)+1$.
If the inequalities (\ref{e1}) and (\ref{e2}) for $2\leq i\leq D-1$ all hold with equality, then
$$n_1+n_2+\cdots+n_{D-1}+n_D
\geq 3\cdot 2^0+3\cdot 2^1+\cdots+3\cdot 2^{D-2}+1
=3\cdot 2^{D-1}-2\stackrel{D>...}{>}n-1,$$
contradicting $(\ref{e3})$.
Hence, we have $n_1<3$ or $n_i<2n_{i-1}$ for some $2\leq i\leq D-1$.
Now, 
\begin{itemize}
\item increasing $n_1$ by $1$ in the first case, 
\item increasing $n_i$ by $1$ in the second case, and 
\item reducing $n_D$ by $1$ in both cases
\end{itemize}
yields a new choice for the non-negative integers $n_1,n_2,\ldots$ 
such that (\ref{e1}), (\ref{e2}), and (\ref{e3}) are satisfied,
but (\ref{e4}) is larger than before, which is a contradiction.
This implies 
$$D\leq \log\left(\frac{n+1}{3}\right)+1\stackrel{n\geq 2}{\leq}\log\left(\frac{n}{2}\right)+1=\log(n).$$
The inequalities (\ref{e1}) and (\ref{e2}) clearly imply $n_i\leq 3\cdot 2^{i-1}$, and, hence,
\begin{eqnarray*}
w_{(G,S\setminus \{ u\})}(u) 
& \leq & \sum\limits_{i=d^*+1}^{\infty}2^{-i+1}n_{i-d^*}\\
& = & \sum\limits_{i=d^*+1}^{D+d^*}2^{-i+1}n_{i-d^*}\\
& \leq & \sum\limits_{i=d^*+1}^{D+d^*}2^{-i+1}\cdot 3\cdot 2^{i-d^*-1}\\
& = & 3\cdot D\cdot 2^{-d^*}\\
& \leq & 3\cdot \log(n)\cdot \frac{1}{2^2\cdot\log(n)}\\
& < & 1,
\end{eqnarray*}
which completes the proof. 
\end{proof}
We believe that Theorem \ref{theorem1} can still be improved,
but that connected subcubic graphs might not have 
exponentially independent sets of linear order.
More precisely,
it seems possible that, 
for every positive $c$, 
there is a connected subcubic graph $G$ of order $n$ 
with $\alpha_e(G)<c\cdot n$;
our next result supports this possibility.

Let $T_{\infty}$ be the infinite cubic tree.
For a vertex $u$ of $T_{\infty}$ and some positive integer $k$,
let the {\it ball $B^{k}(u)$ of radius $k$ around $u$
be the set of vertices $v$ of $T_{\infty}$ with ${\rm dist}_{T_{\infty}}(u,v)\leq k$},
in particular, 
$$|B^{ k}(u)|=1+3\cdot 2^0+\cdots+3\cdot 2^{k-1}=3\cdot 2^k-2.$$
For a set $S$ of vertices of $T_{\infty}$, let 
$$f(u,S)=\limsup\limits_{k\to\infty}\frac{|S\cap B^{ k}(u)|}{|B^{ k}(u)|},$$
which measures the asymptotic density of $S$ within $T_{\infty}$.

\begin{theorem}\label{theoreminf}
If $S$ is a set of vertices of the infinite cubic tree $T_{\infty}$,
and $f(u^*,S)>0$ for some vertex $u^*$ of $T_{\infty}$,
then $S$ is not an exponentially independent set.
\end{theorem}
\begin{proof}
Suppose, for a contradiction, that $S$ is an exponentially independent set.
For brevity, we write $f(\cdot)$ instead of $f(\cdot,S)$. 
For a vertex $u$ and some positive integer $k$, let 
$$f_k(u)=\frac{|S\cap B^{ k}(u)|}{|B^{ k}(u)|},$$
that is, $f(u)=\limsup\limits_{k\to\infty}f_k(u)$.
If $u$ and $v$ are two vertices in $T_{\infty}$ with distance $d$,
then 
$$S\cap B^{ k}(u)\subseteq S\cap B^{ (k+d)}(v),$$
which implies 
$$f_{k+d}(v)\geq \frac{3\cdot 2^k-2}{3\cdot 2^{k+d}-2}f_k(u),$$
and, hence, taking the limit superior,
$$f(v)\geq \frac{1}{2^d}f(u).$$
Since $f(u^*)>0$, it follows that $f(u)>0$ for every vertex $u$ of $T_{\infty}$.
By definition, $f(u)\leq 1$ for every vertex $u$ of $T_{\infty}$.
Therefore, applying the following claim 
$\left\lfloor\log\left(\frac{1}{f(u)}\right)\right\rfloor+1$ times
to some vertex $u$ in $S$ yields a contradiction, completing the proof.

\begin{claim}\label{claim1}
For every vertex $u$ in $S$, there is a vertex $v$ in $S$ with $f(v)\geq 2f(u)$.
\end{claim}
\begin{proof}[Proof of Claim \ref{claim1}]
We consider $T_{\infty}$ to be rooted in $u$.
For a vertex $x$ distinct from $u$ and a non-negative integer $k$, 
let $\vec{B}^{k}(x)$ be the set containing $x$ and all descendants of $x$ in $B^{ k}(x)$,
in particular, $|\vec{B}^{k}(x)|=2^{k+1}-1$.
Furthermore, let 
$$\vec{f}_k(x)=\frac{|S\cap \vec{B}^{k}(x)|}{|\vec{B}^{k}(x)|},$$
and 
$$\vec{f}(x)=\limsup\limits_{k\to\infty}\vec{f}_k(x).$$
Since 
$S\cap \vec{B}^{k}(x)\subseteq S\cap B^{ k}(x),$
we have 
$$f_k(x)\geq \frac{2^{k+1}-1}{3\cdot 2^k-2}\vec{f}_k(x).$$
Taking the limit superior, we obtain
$$f(x)\geq \frac{2}{3}\vec{f}(x).$$
For a positive integer $i$, let $X_i$ be the set of vertices $x$ in $S$
with ${\rm dist}_{(T_{\infty},S\setminus \{ u\})}(u,x)=i$.
Let $X=\bigcup\limits_{i=1}^{\infty}X_i$.
Note that, by definition,
no element of $X$ is a descendant 
of another element of $X$.
Since $S$ is exponentially independent, we have
$$w_{(T_{\infty},S\setminus \{ u\})}(u)=\sum\limits_{i=1}^{\infty}\sum\limits_{x\in X_i}2^{-i+1}<1.$$
Let $k$ be a positive integer.
Since every vertex in $S\cap B^{ k}(u)$ distinct from $u$
is either in $X$ or is a descendant 
of exactly one vertex in $X$, we have
$$|S\cap B^{ k}(u)|=
1+\sum\limits_{i=1}^k\sum\limits_{x\in X_i}|S\cap \vec{B}^{(k-i)}(x)|,$$
and, hence,
$$f_k(u)=\frac{1}{3\cdot 2^k-2}\left(1+\sum\limits_{i=1}^{\infty}\sum\limits_{x\in X_i}
\left(2^{k-i+1}-1\right)\vec{f}_{k-i}(x)\right),$$
where, for notational simplicity, we set $\vec{f}_k(x)=0$ for negative integers $k$.
By the subadditivity of the limit superior, we obtain
\begin{eqnarray}\label{elimsup}
f(u)\leq \frac{1}{3}\sum\limits_{i=1}^{\infty}\sum\limits_{x\in X_i}2^{-i+1}\cdot \vec{f}(x)
\end{eqnarray}
Now, suppose, for a contradiction, that 
$f(x)<2f(u)$ for every $x$ in $X$.
Since $f(x)\geq \frac{2}{3}\vec{f}(x)$,
this implies $\frac{\vec{f}(x)}{3f(u)}<1$ for every $x$ in $X$, and we obtain
\begin{eqnarray*}
w_{(T_{\infty},S\setminus \{ u\})}(u)&=&\sum\limits_{i=1}^{\infty}\sum\limits_{x\in X_i}2^{-i+1}\\
&>&\sum\limits_{i=1}^{\infty}\sum\limits_{x\in X_i}\frac{2^{-i+1}\cdot \vec{f}(x)}{3f(u)}\\
&=&\frac{1}{f(u)}\left(\frac{1}{3}\sum\limits_{i=1}^{\infty}\sum\limits_{x\in X_i}2^{-i+1}\cdot \vec{f}(x)\right)\\
&\stackrel{(\ref{elimsup})}{\geq} &1,
\end{eqnarray*}
which is a contradiction.
\end{proof}
As observed before the claim, the proof is complete.
\end{proof}
Intuitively speaking, 
the condition ``$f(u^*,S)>0$'' in Theorem \ref{theoreminf}
expresses that the set $S$ is (asymptotically) well spread within $T_{\infty}$.
A similar statement for finite graphs
that can easily be shown is the following:
{\it If $p\in (0,1]$,
$T(k)$ is the perfect binary tree of depth $k$,
and $S$ is a random set of vertices of $T(k)$
containing its root 
and containing every other vertex of $T(k)$ 
independently at random with probability $p$, then 
$$\lim_{k\to \infty}\mathbb{P}\left[\mbox{$S$ is exponentially independent in $T(k)$}\right]=0.$$}
These statements suggest that a (proof) method 
that establishes the existence 
of exponentially independent sets of linear order
in general subcubic graphs should sometimes select 
sets that are not well spread.
In a perfect binary tree, for instance,
the only optimum choice is to select all leaves \cite{jr}.
Note that the proof method of Theorem \ref{theorem1}
tends to generate well spread sets $S$;
the sets constructed there have the property that,
for every vertex of the considered graph $G$,
some vertex from $S$ is within distance $O(\log(\log(n)))$.

Theorem \ref{theoreminf} also indicates 
that it is the exponential expansion that is problematic.
Variating the proof of Theorem \ref{theorem1}, 
our next result establishes a linear lower bound 
on the exponential independence number of subcubic graphs
whose expansion is only mildly restricted.
Note that,
if $G$ is a subcubic graph, $u$ is a vertex of $G$, and $d$ is some positive integer,
then the {\it $d$-neighborhood $N^d_G(u)$ of $u$ in $G$},
that is, the set
$$N^d_G(u)=\{ v\in V(G):{\rm dist}_G(u,v)=d\},$$
contains at most $3\cdot 2^{d-1}$ vertices.

\begin{theorem}\label{theorem1b}
For every positive integer $d$, there is a positive integer $d^*$ with the following property:
If $G$ is a subcubic graph $G$ of order $n$ such that, 
for every vertex $u$ of $G$, we have
$$|N^d_G(u)|\leq 3\cdot 2^{d-1}-1,$$
and $S$ is a set of vertices of $G$ such that the distance in $G$
between any two vertices from $S$ is more than $2d^*$,
then $S$ is an exponentially independent set in $G$.
In particular, 
$$\alpha_e(G)\geq c_d\cdot n$$
for some positive constant $c_d$ depending only on $d$.
\end{theorem}
\begin{proof}
Let $\epsilon>0$ be such that $\left(2^{2d}-1\right)^{\frac{1}{2d}}=2-\epsilon$.
Let $d^*$ be such that 
$\epsilon\cdot (2-\epsilon)^{d^*}>3\cdot 2^{2d+1}$.
Let $u$ be any vertex in $S$.
Define the tree $T$ rooted in $u$, 
and the values $n_i$ and $\alpha_i$ exactly as in the proof of Theorem \ref{theorem1}.
Again, we need to show that $\sum\limits_{i=1}^{\infty}2^{-i+1}\alpha_i<1$.

Exactly as in the proof of Theorem \ref{theorem1}, we obtain
\begin{eqnarray*}
n_i &\leq &
\begin{cases}
3 & \mbox{for $i=1$,}\\
2n_{i-1}& \mbox{ for $i\geq 2$,}
\end{cases}\\[3mm]
\alpha_i & \leq &
\begin{cases}
0 & \mbox{for $i\leq 2d^*$, and}\\
n_{i-d^*} & \mbox{ for $i\geq 2d^*+1$.}
\end{cases}
\end{eqnarray*}
Now, let $i$ be some positive integer, and let $v$ be a vertex at depth $i$ in $T$
that does not belong to $S$.
If $v$ has $2^{2d}$ descendants in $T$ at depth $i+2d$,
and $w$ is a descendant of $u$ at depth $i+d$,
then $|N^d_G(w)|=3\cdot 2^{d-1}$,
which is a contradiction.
Hence, $v$ has at most $2^{2d}-1$ descendants at depth $i+2d$, 
which implies
$$n_{i+2d} \leq \left(2^{2d}-1\right)n_i=(2-\epsilon)^{2d}n_i.$$
It follows that 
\begin{eqnarray*}
n_i & \leq & n_1\cdot 2^{(i-1)-2d\cdot \lfloor\frac{i-1}{2d}\rfloor}
\cdot \left(2^{2d}-1\right)^{\lfloor\frac{i-1}{2d}\rfloor}
\leq 3\cdot 2^{2d}\cdot (2-\epsilon)^{i-1}.
\end{eqnarray*}
By the choice of $d^*$, we obtain
\begin{eqnarray*}
\sum\limits_{i=1}^{\infty}2^{-i+1}\alpha_i
& \leq & \sum\limits_{i=d^*+1}^{\infty}2^{-i+1}n_{i-d^*}\\
& \leq & \sum\limits_{i=d^*+1}^{\infty}2^{-i+1}\cdot 3\cdot 2^{2d}\cdot (2-\epsilon)^{i-d^*-1}\\
& \leq & \frac{3\cdot 2^{2d}}{(2-\epsilon)^{d^*}}
\sum\limits_{i=0}^{\infty}\left(\frac{2-\epsilon}{2}\right)^i\\
&=& \frac{3\cdot 2^{2d+1}}{\epsilon\cdot (2-\epsilon)^{d^*}}\\
&<& 1,
\end{eqnarray*}
which completes the proof.
\end{proof}
We briefly consider graphs of maximum degrees larger than $3$.
For integers $\Delta\geq 3$ and $d\geq 0$, let $T(\Delta,d)$ be the rooted tree 
in which every vertex that is not a leaf has degree $\Delta$,
and all leaves have depth $d$.

\begin{fact}\label{factDelta6}
There is no positive constant $c$ such that $\alpha_e(T(6,d))\geq c\cdot n(T(6,d))+1$
for every positive integer $d$.
\end{fact}
\begin{proof}
Suppose, for a contradiction, that such a value $c$ exists.
Let $d_c$ be such that
$$2^{-2d_c+1}\cdot\left(\frac{c}{2}\cdot 6\cdot 5^{d_c-1}-1\right)>1,$$
and let $D_c$ be such that
$$\frac{c}{2}\cdot \left(\frac{3}{2}\cdot 5^{D_c}-\frac{1}{2}\right)
>\frac{3}{2}\cdot 5^{d_c}-\frac{1}{2}.$$
Let $d\geq D_c$.
Let $S$ be an exponentially independent set in $T(6,d)$
of order at least $$c\cdot n(T(6,d))+1.$$
Let $S_i$ be the elements of $S$ at depth $i$ in $T(6,d)$.
Since there are exactly $6\cdot 5^{i-1}$ vertices in $T(6,d)$ of depth $i$,
the pigeonhole principle implies the existence of some positive integer $k$ 
such that $$|S_k|\geq \frac{c}{2}\cdot 6\cdot 5^{k-1}.$$
We choose $k$ as large as possible,
and suppose, for a contradiction,
that $k<d_c$.
In this case,
\begin{eqnarray*}
|S| & \leq & \frac{c}{2}\cdot n(T(6,d))+n(T(6,d_c))\\
& = & \frac{c}{2}\cdot \left(\frac{3}{2}\cdot 5^{d}-\frac{1}{2}\right)
+\left(\frac{3}{2}\cdot 5^{d_c}-\frac{1}{2}\right)\\
& \stackrel{d\geq D_c}{<} & c\cdot \left(\frac{3}{2}\cdot 5^{d}-\frac{1}{2}\right)\\
& = & c\cdot n(T(6,d)),
\end{eqnarray*}
which is a contradiction.
If follows that $k\geq d_c$.
As shown in \cite{jr},
every subset of an exponentially independent set is exponentially independent.
Therefore, the set $S_k$ is exponentially independent.
For a vertex $u\in S_k$ though, we obtain
\begin{eqnarray*}
w_{(T(6,d),S_k\setminus \{ u\})}(u) & \geq & 2^{-2k+1}\cdot (|S_k|-1)\\
& \geq & 2^{-2k+1}\cdot\left(\frac{c}{2}\cdot 6\cdot 5^{k-1}-1\right)\\
& \stackrel{k\geq d_c}{\geq} &
2^{-2d_c+1}\cdot\left(\frac{c}{2}\cdot 6\cdot 5^{d_c-1}-1\right)\\
&>&1,
\end{eqnarray*}
which is a contradiction.
This completes the proof.
\end{proof}
One might guess that Fact \ref{factDelta6}
also holds for $T(4,d)$, that is, 
that already maximum degree $4$ should suffice.
This intuition is misleading though.
In view of the following fact, 
it seems even conceivable that trees of maximum degree at most $4$
have exponentially independent sets of linear order.

\begin{fact}\label{factDelta4}
$\alpha_e(T(4,d))\geq \frac{2}{3^3}\cdot \left(n(T(4,d))+1\right)$ for $d\geq 3$.
\end{fact}
\begin{proof}
For notational simplicity, we consider $T(4,d+2)$ for some $d\geq 1$.
Let the set $S$ contain exactly one grandchild of every vertex of $T(4,d+2)$ 
at depth $d$,
that is, all elements of $S$ have maximum depth $d+2$, and 
$$
|S|=4\cdot 3^{d-1}
= \frac{2}{3^3}\cdot 2\cdot 3^{d+2}
= \frac{2}{3^3}\cdot \left(n(T(4,d+2))+1\right).$$
Our goal is to show that $S$ is exponentially independent.
For $u\in S$, we obtain
\begin{eqnarray*}
w_{(T(4,d+2),S\setminus \{ u\})}(u) 
& = & 
\frac{1}{2^4}
\left(
3^d\cdot 2^{-2d+1}
+2\cdot 3^{d-2}\cdot 2^{-2(d-1)+1}
+\cdots
+2\cdot 3^{0}\cdot 2^{-2(d-(d-1))+1}
\right)\\
& \leq & 
\frac{1}{2^4}
\left(
3^d\cdot 2^{-2d+1}
+3^{d-1}\cdot 2^{-2(d-1)+1}
+\cdots
+3^{1}\cdot 2^{-2(d-(d-1))+1}
\right)\\
&=& \frac{1}{8}\cdot \sum_{i=1}^d\left(\frac{3}{4}\right)^i\\
& < & \frac{1}{2},
\end{eqnarray*}
which completes the proof.
\end{proof}
A natural class of subcubic graphs with exponentially independent sets of linear order are trees.
In fact, J\"{a}ger et al.~\cite{jr} showed the linear lower bound
\begin{eqnarray}
\alpha_e(T)\geq \frac{2n+8}{13}\label{e5}
\end{eqnarray}
for every subcubic tree $T$ of order $n$.
Even for the restricted class of subcubic trees, 
the smallest possible value of the exponential independence number 
as well as its computational complexity are unknown.
Note that, in contrast to that, Bessy et al.~\cite{bor2} 
described a polynomial time algorithm that determines the exponential domination number 
of a given subcubic tree.

Our next goal is to improve (\ref{e5}).
Actually, we believe that 
$\alpha_e(T)\geq \frac{n+O(1)}{3}$ for every subcubic tree $T$ of order $n$.
The trees $T_k$ illustrated in Figure \ref{fig1} 
show that such a bound would be best possible
up to the additive constant, cf.~Fact \ref{fact1}.

\begin{figure}[H]
\begin{center}
\unitlength 0.8mm 
\linethickness{0.4pt}
\ifx\plotpoint\undefined\newsavebox{\plotpoint}\fi 
\begin{picture}(156,31)(0,0)
\put(15,5){\circle*{1.5}}
\put(145,5){\circle*{1.5}}
\put(25,5){\circle*{1.5}}
\put(135,5){\circle*{1.5}}
\put(35,5){\circle*{1.5}}
\put(115,5){\circle*{1.5}}
\put(55,5){\circle*{1.5}}
\put(105,5){\circle*{1.5}}
\put(65,5){\circle*{1.5}}
\put(95,5){\circle*{1.5}}
\put(75,5){\circle*{1.5}}
\put(85,5){\circle*{1.5}}
\put(15,15){\circle*{1.5}}
\put(145,15){\circle*{1.5}}
\put(25,15){\circle*{1.5}}
\put(135,15){\circle*{1.5}}
\put(35,15){\circle*{1.5}}
\put(115,15){\circle*{1.5}}
\put(55,15){\circle*{1.5}}
\put(105,15){\circle*{1.5}}
\put(65,15){\circle*{1.5}}
\put(95,15){\circle*{1.5}}
\put(75,15){\circle*{1.5}}
\put(85,15){\circle*{1.5}}
\put(15,25){\circle*{1.5}}
\put(145,25){\circle*{1.5}}
\put(25,25){\circle*{1.5}}
\put(135,25){\circle*{1.5}}
\put(35,25){\circle*{1.5}}
\put(115,25){\circle*{1.5}}
\put(55,25){\circle*{1.5}}
\put(105,25){\circle*{1.5}}
\put(65,25){\circle*{1.5}}
\put(95,25){\circle*{1.5}}
\put(75,25){\circle*{1.5}}
\put(85,25){\circle*{1.5}}
\put(15,5){\line(0,1){20}}
\put(145,5){\line(0,1){20}}
\put(25,5){\line(0,1){20}}
\put(135,5){\line(0,1){20}}
\put(35,5){\line(0,1){20}}
\put(115,5){\line(0,1){20}}
\put(55,5){\line(0,1){20}}
\put(105,5){\line(0,1){20}}
\put(65,5){\line(0,1){20}}
\put(95,5){\line(0,1){20}}
\put(75,5){\line(0,1){20}}
\put(85,5){\line(0,1){20}}
\put(15,25){\line(1,0){10}}
\put(145,25){\line(-1,0){10}}
\put(25,25){\line(1,0){10}}
\put(115,25){\line(-1,0){10}}
\put(55,25){\line(1,0){10}}
\put(105,25){\line(-1,0){10}}
\put(65,25){\line(1,0){10}}
\put(95,25){\line(-1,0){10}}
\put(75,25){\line(1,0){10}}
\put(5,15){\circle*{1.5}}
\put(155,15){\circle*{1.5}}
\put(5,5){\circle*{1.5}}
\put(155,5){\circle*{1.5}}
\put(5,5){\line(0,1){10}}
\put(155,5){\line(0,1){10}}
\put(5,15){\line(1,1){10}}
\put(155,15){\line(-1,1){10}}
\put(70,2){\framebox(10,26)[cc]{}}
\put(80,2){\framebox(10,26)[cc]{}}
\put(10,2){\framebox(10,26)[cc]{}}
\put(140,2){\framebox(10,26)[cc]{}}
\put(90,2){\framebox(10,26)[cc]{}}
\put(85,31){\makebox(0,0)[cc]{$V_i$}}
\put(15,31){\makebox(0,0)[cc]{$V_1$}}
\put(145,31){\makebox(0,0)[cc]{$V_k$}}
\put(95,31){\makebox(0,0)[cc]{$V_{i+1}$}}
\put(75,31){\makebox(0,0)[cc]{$V_{i-1}$}}
\put(45,15){\makebox(0,0)[cc]{$\ldots$}}
\put(125,15){\makebox(0,0)[cc]{$\ldots$}}
\put(35,25){\line(1,0){5}}
\put(115,25){\line(1,0){5}}
\put(55,25){\line(-1,0){5}}
\put(135,25){\line(-1,0){5}}
\end{picture}
\end{center}
\caption{The illustrated tree $T_k$ has order $n(T_k)=3k+4$.}\label{fig1}
\end{figure}
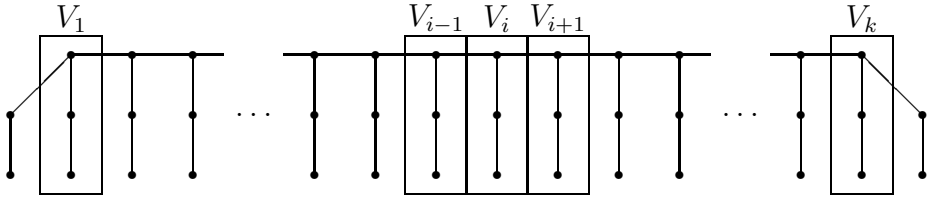

\begin{fact}\label{fact1}
$\alpha_e(T_k)=\frac{n(T_k)+2}{3}$.
\end{fact}
\begin{proof}
Let $S$ be a maximum exponentially independent set of $T_k$
such that the number of indices $i$ with $|V_i\cap S|=0$ is as small as possible,
where the sets $V_i$ are as indicated in Figure \ref{fig1}.
If $|V_i\cap S|\geq 2$ for some $i$ with $2\leq i\leq k-1$, 
then $S$ contains the endvertex as well as the vertex $u_i$ of degree $3$ from $V_i$.
Since $S$ is exponentially independent, it follows that 
$|V_{i-1}\cap S|=0$ or $|V_{i+1}\cap S|=0$.
Removing $u_i$ and adding the endvertex from $V_{i-1}$ if $|V_{i-1}\cap S|=0$,
or the endvertex from $V_{i+1}$ otherwise, yields an exponentially independent set
contradicting the choice of $S$. 
Similar arguments concerning $V_1$ and $V_k$ easily imply
that the set of all endertices is a maximum exponentially independent set,
and, hence, we have $\alpha_e(T_k)=k+2=\frac{n+2}{3}$.
\end{proof}
If $T$ is a subcubic tree of order $n$ that has no vertex of degree $2$, 
then it is easy to see, cf.~\cite{jr},
that the set of all but one of the endvertices of $T$ 
is an exponentially independent set in $T$ of order $\frac{n}{2}$. Therefore, we may assume that 
there is at least one vertex of degree $2$.

\begin{theorem}\label{theorem4}
If $T$ is a subcubic tree of order $n$ that has at least one vertex of degree $2$, 
then $T$ has an exponentially independent set $S$ containing all endvertices of $T$
such that 
$$|S|\geq \frac{n+3}{4},$$
in particular, $\alpha_e(T)\geq \frac{n+3}{4}.$
\end{theorem}
\begin{proof}
A set $S$ of vertices of a subcubic tree $T$ of order $n$ is called {\it good in $T$}
if $S$ is exponentially independent in $T$,
contains all endvertices of $T$, and has order at least $\frac{n+3}{4}$.
Let $T$ be a counterexample 
to the statement of the theorem
that is of minimum order,
in particular, $T$ has no good set.
It is easy to see that $T$ is not a path;
in fact, the path of order $5$ satisfies $\alpha_e(P_5)=2=\frac{5+3}{4}$.

\setcounter{claim}{0}
\begin{claim}\label{c1}
No vertex is adjacent to two endvertices.
\end{claim}
\begin{proof}[Proof of Claim \ref{c1}]
Suppose, for a contradiction, that the vertex $v$ is adjacent to the two endvertices $u_1$ and $u_2$.
Since $T$ is not a path of order $3$,
the tree $T'=T-\{ u_1,u_2\}$ is a subcubic tree that has at least one vertex of degree $2$.
By the choice of $T$,
the tree $T'$ has a good set $S'$.
Clearly, $v\in S'$.
The set $S=(S'\setminus \{ v\})\cup \{ u_1,u_2\}$ contains all endvertices of $T$,
and has order more than $\frac{n+3}{4}$.
Since $S'$ is exponentially independent in $T'$, we obtain
\begin{eqnarray*}
w_{(T,S\setminus \{ u_i\})}(u_i)&=&\frac{1}{2}+\frac{1}{2}\cdot w_{(T',S'\setminus \{ v\})}(v)<1
\mbox{ for $i\in \{ 1,2\}$, and}\\[3mm]
w_{(T,S\setminus \{ x\})}(x)&=&w_{(T',S'\setminus \{ x\})}(x)<1
\mbox{ for every vertex $x$ in $S\setminus \{ u_1,u_2\}=S'\setminus \{ v\}$},
\end{eqnarray*}
that is, $S$ is good in $T$, 
contradicting the choice of $T$.
\end{proof}
Let $P:w_1w_2\ldots w_d$ be a longest path in $T$.
Since $T$ is not a path, there is a smallest index $k$ such that $w_k$ has degree $3$ in $T$.
Claim \ref{c1} implies $k\geq 3$.

\begin{claim}\label{c2}
$k\leq 4$.
\end{claim}
\begin{proof}[Proof of Claim \ref{c2}]
Suppose, for a contradiction, that $k\geq 5$.
The tree $T'=T-\{ w_1,w_2,w_3\}$ is subcubic.
If $T'$ has no vertex of degree $2$, 
then, since $T$ is not a path, 
some vertex of $T$ is adjacent to two endvertices,
which contradicts Claim \ref{c1}.
Hence, $T'$ has at least one vertex of degree $2$.
By the choice of $T$,
the tree $T'$ has a good set $S'$.
Clearly, $w_4\in S'$.
The set $S=(S'\setminus \{ w_4\})\cup \{ w_1,w_3\}$ contains all endvertices of $T$,
and has order more than $\frac{n+3}{4}$.
Since $S'$ is exponentially independent in $T'$, we obtain
\begin{eqnarray*}
w_{(T,S\setminus \{ w_1\})}(w_1)&=&\frac{1}{2}<1,\\[3mm]
w_{(T,S\setminus \{ w_3\})}(w_3)&=&\frac{1}{2}+\frac{1}{2}\cdot w_{(T',S'\setminus \{ w_4\})}(w_4)<1,\mbox{ and }\\[3mm]
w_{(T,S\setminus \{ x\})}(x)&<&w_{(T',S'\setminus \{ x\})}(x)<1
\mbox{ for every vertex $x$ in $S\setminus \{ w_1,w_3\}=S'\setminus \{ w_4\}$},
\end{eqnarray*}
that is, $S$ is good in $T$, 
contradicting the choice of $T$.
\end{proof}

\begin{claim}\label{c3}
$k=4$.
\end{claim}
\begin{proof}[Proof of Claim \ref{c3}]
Suppose, for a contradiction, that $k=3$.
Let $w_2'$ be the neighbor of $w_3$ distinct from $w_2$ and $w_4$.
By the choice of $P$ and Claim \ref{c1}, the degree of $w_2'$ is at most $2$.
First, we assume that $w_2'$ has degree $1$.
By Claim \ref{c1}, $w_4$ is not an endvertex,
and the tree $T'=T-\{ w_1,w_2,w_2'\}$ has a vertex of degree $2$.
By the choice of $T$,
the tree $T'$ has a good set $S'$.
Arguing as above, 
it follows that $S=(S'\setminus \{ w_3\})\cup \{ w_1,w_2'\}$ 
is a good set in $T$.
Hence, we may assume that $w_2'$ has degree $2$.
By the choice of $P$, the neighbor $w_1'$ of $w_2'$ distinct from $w_3$ is an endvertex.
The tree $T'=T-\{ w_1,w_1',w_2'\}$ has the vertex $w_3$ of degree $2$.
By the choice of $T$,
the tree $T'$ has a good set $S'$.
Arguing as above, 
it follows that $S=(S'\setminus \{ w_2\})\cup \{ w_1,w_1'\}$ 
is a good set in $T$.
This contradiction completes the proof of the claim.
\end{proof}
Now, we are in a position to derive a final contradiction.
Let $w_3'$ be the neighbor of $w_4$ distinct from $w_3$ and $w_5$.
By the claims and symmetry,
we may assume that the component $K$ of $T-w_4$ that contains $w_3'$ 
is a subtree of the path $w_3'w_2'w_1'$ containing $w_3'$.
If $n(K)=3$, then let $T'=T-\{ w_1,w_1',w_2',w_3'\}$,
if $n(K)=2$, then let $T'=T-\{ w_1,w_2,w_2',w_3'\}$, and
if $n(K)=1$, then let $T'=T-\{ w_1,w_2,w_3,w_3'\}$.
In the first two cases, $w_4$ has degree $2$ in $T'$ by construction,
and in the third case, Claim \ref{c1} implies that $w_5$ is not an endvertex,
and that $T'$ has a vertex of degree $2$. 
Arguing as above for a good set in $T'$ 
yields a good set in $T$.
This final contradiction completes the proof.
\end{proof}
In order to obtain large exponentially independent sets in trees,
it seems natural to give preference to the endvertices.
In fact, this is what happens for the trees in Figure \ref{fig1}.
For the proof of Theorem \ref{theorem4},
it was actually technically important for the argument 
to consider only exponentially independent sets containing all endvertices.
Nevertheless, there are subcubic trees $T_k'$ of arbitrarily large order $n$,
illustrated in Figure \ref{fig2},
that do not have exponentially independent sets of order $\frac{n+O(1)}{3}$
containing all endvertices.
The reason, why sometimes non-endvertices are preferable over endvertices
are blocking effects.

\begin{figure}[H]
\begin{center}
\unitlength 0.6mm 
\linethickness{0.4pt}
\ifx\plotpoint\undefined\newsavebox{\plotpoint}\fi 
\begin{picture}(261,77)(0,0)
\put(15,5){\circle*{2}}
\put(206,5){\circle*{2}}
\put(45,5){\circle*{2}}
\put(236,5){\circle*{2}}
\put(95,5){\circle*{2}}
\put(125,5){\circle*{2}}
\put(155,5){\circle*{2}}
\put(25,25){\circle*{2}}
\put(216,25){\circle*{2}}
\put(55,25){\circle*{2}}
\put(246,25){\circle*{2}}
\put(105,25){\circle*{2}}
\put(135,25){\circle*{2}}
\put(165,25){\circle*{2}}
\put(35,15){\circle*{2}}
\put(226,15){\circle*{2}}
\put(65,15){\circle*{2}}
\put(256,15){\circle*{2}}
\put(115,15){\circle*{2}}
\put(145,15){\circle*{2}}
\put(175,15){\circle*{2}}
\put(35,55){\circle*{2}}
\put(226,55){\circle*{2}}
\put(65,55){\circle*{2}}
\put(256,55){\circle*{2}}
\put(115,55){\circle*{2}}
\put(145,55){\circle*{2}}
\put(175,55){\circle*{2}}
\put(15,15){\circle*{2}}
\put(206,15){\circle*{2}}
\put(45,15){\circle*{2}}
\put(236,15){\circle*{2}}
\put(95,15){\circle*{2}}
\put(125,15){\circle*{2}}
\put(155,15){\circle*{2}}
\put(25,35){\circle*{2}}
\put(216,35){\circle*{2}}
\put(55,35){\circle*{2}}
\put(246,35){\circle*{2}}
\put(105,35){\circle*{2}}
\put(135,35){\circle*{2}}
\put(165,35){\circle*{2}}
\put(35,25){\circle*{2}}
\put(226,25){\circle*{2}}
\put(65,25){\circle*{2}}
\put(256,25){\circle*{2}}
\put(115,25){\circle*{2}}
\put(145,25){\circle*{2}}
\put(175,25){\circle*{2}}
\put(35,65){\circle*{2}}
\put(226,65){\circle*{2}}
\put(65,65){\circle*{2}}
\put(256,65){\circle*{2}}
\put(115,65){\circle*{2}}
\put(145,65){\circle*{2}}
\put(175,65){\circle*{2}}
\put(15,25){\circle*{2}}
\put(206,25){\circle*{2}}
\put(45,25){\circle*{2}}
\put(236,25){\circle*{2}}
\put(95,25){\circle*{2}}
\put(125,25){\circle*{2}}
\put(155,25){\circle*{2}}
\put(25,45){\circle*{2}}
\put(216,45){\circle*{2}}
\put(55,45){\circle*{2}}
\put(246,45){\circle*{2}}
\put(105,45){\circle*{2}}
\put(135,45){\circle*{2}}
\put(165,45){\circle*{2}}
\put(35,35){\circle*{2}}
\put(226,35){\circle*{2}}
\put(65,35){\circle*{2}}
\put(256,35){\circle*{2}}
\put(115,35){\circle*{2}}
\put(145,35){\circle*{2}}
\put(175,35){\circle*{2}}
\put(35,75){\circle*{2}}
\put(226,75){\circle*{2}}
\put(65,75){\circle*{2}}
\put(256,75){\circle*{2}}
\put(115,75){\circle*{2}}
\put(145,75){\circle*{2}}
\put(175,75){\circle*{2}}
\put(25,55){\circle*{2}}
\put(216,55){\circle*{2}}
\put(55,55){\circle*{2}}
\put(246,55){\circle*{2}}
\put(105,55){\circle*{2}}
\put(135,55){\circle*{2}}
\put(165,55){\circle*{2}}
\put(35,55){\line(-1,0){10}}
\put(226,55){\line(-1,0){10}}
\put(65,55){\line(-1,0){10}}
\put(256,55){\line(-1,0){10}}
\put(115,55){\line(-1,0){10}}
\put(145,55){\line(-1,0){10}}
\put(175,55){\line(-1,0){10}}
\put(25,55){\line(0,-1){30}}
\put(216,55){\line(0,-1){30}}
\put(55,55){\line(0,-1){30}}
\put(246,55){\line(0,-1){30}}
\put(105,55){\line(0,-1){30}}
\put(135,55){\line(0,-1){30}}
\put(165,55){\line(0,-1){30}}
\put(25,45){\line(1,-1){10}}
\put(216,45){\line(1,-1){10}}
\put(55,45){\line(1,-1){10}}
\put(246,45){\line(1,-1){10}}
\put(105,45){\line(1,-1){10}}
\put(135,45){\line(1,-1){10}}
\put(165,45){\line(1,-1){10}}
\put(35,35){\line(0,-1){20}}
\put(226,35){\line(0,-1){20}}
\put(65,35){\line(0,-1){20}}
\put(256,35){\line(0,-1){20}}
\put(115,35){\line(0,-1){20}}
\put(145,35){\line(0,-1){20}}
\put(175,35){\line(0,-1){20}}
\put(25,35){\line(-1,-1){10}}
\put(216,35){\line(-1,-1){10}}
\put(55,35){\line(-1,-1){10}}
\put(246,35){\line(-1,-1){10}}
\put(105,35){\line(-1,-1){10}}
\put(135,35){\line(-1,-1){10}}
\put(165,35){\line(-1,-1){10}}
\put(15,25){\line(0,-1){20}}
\put(206,25){\line(0,-1){20}}
\put(45,25){\line(0,-1){20}}
\put(236,25){\line(0,-1){20}}
\put(95,25){\line(0,-1){20}}
\put(125,25){\line(0,-1){20}}
\put(155,25){\line(0,-1){20}}
\put(35,75){\line(0,-1){20}}
\put(226,75){\line(0,-1){20}}
\put(65,75){\line(0,-1){20}}
\put(256,75){\line(0,-1){20}}
\put(115,75){\line(0,-1){20}}
\put(145,75){\line(0,-1){20}}
\put(175,75){\line(0,-1){20}}
\put(13,3){\framebox(24,74)[cc]{}}
\put(204,3){\framebox(24,74)[cc]{}}
\put(43,3){\framebox(24,74)[cc]{}}
\put(234,3){\framebox(24,74)[cc]{}}
\put(93,3){\framebox(24,74)[cc]{}}
\put(123,3){\framebox(24,74)[cc]{}}
\put(153,3){\framebox(24,74)[cc]{}}
\put(35,55){\line(1,0){5}}
\put(226,55){\line(1,0){5}}
\put(65,55){\line(1,0){5}}
\put(115,55){\line(1,0){5}}
\put(145,55){\line(1,0){5}}
\put(175,55){\line(1,0){5}}
\put(55,55){\line(-1,0){15}}
\put(246,55){\line(-1,0){15}}
\put(105,55){\line(-1,0){15}}
\put(216,55){\line(-1,0){15}}
\put(135,55){\line(-1,0){15}}
\put(165,55){\line(-1,0){15}}
{\footnotesize
\put(135,58){\makebox(0,0)[cc]{$a_i$}}
\put(165,58){\makebox(0,0)[cc]{$a_{i+1}$}}
\put(112,51){\makebox(0,0)[cc]{$b_{i-1}$}}
\put(145,51){\makebox(0,0)[cc]{$b_i$}}
\put(142,33){\makebox(0,0)[cc]{$c_i$}}
}
\put(25,70){\makebox(0,0)[cc]{$V_1$}}
\put(216,70){\makebox(0,0)[cc]{$V_{k-1}$}}
\put(55,70){\makebox(0,0)[cc]{$V_2$}}
\put(246,70){\makebox(0,0)[cc]{$V_k$}}
\put(105,70){\makebox(0,0)[cc]{$V_{i-1}$}}
\put(135,70){\makebox(0,0)[cc]{$V_i$}}
\put(165,70){\makebox(0,0)[cc]{$V_{i+1}$}}
\put(80,55){\makebox(0,0)[cc]{$\ldots$}}
\put(191,55){\makebox(0,0)[cc]{$\ldots$}}
\put(15,5){\circle{3.5}}
\put(35,15){\circle{3.5}}
\put(15,25){\circle{3.5}}
\put(35,75){\circle{3.5}}
\put(33.25,33.25){\framebox(3.5,3.5)[cc]{}}
\end{picture}
\end{center}
\caption{The illustrated tree $T'_k$ has order $n(T'_k)=13k$.}\label{fig2}
\end{figure}
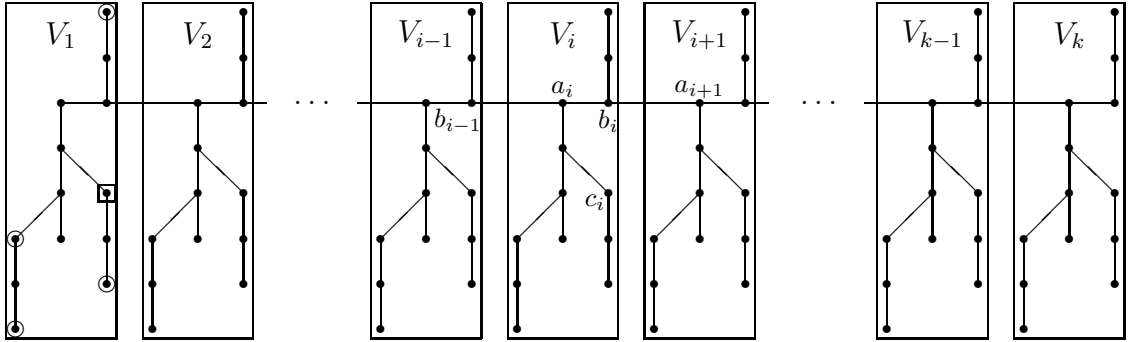

\begin{fact}\label{fact2}
If $S$ is an exponentially independent set in $T_k'$ 
containing all endvertices of $T_k'$, then 
$$|S|\leq \frac{4}{13}n(T_k')+O(1).$$
\end{fact}
\begin{proof}
Let $L_i$ be the set of the four endvertices in $V_i$, 
and let $S_i=S\cap V_i$, 
where the sets $V_i$ are as indicated in Figure \ref{fig2}.
Since $L_i\subseteq S_i$,
it follows easily that $S_i\setminus L_i\subseteq \{ a_i,b_i,c_i\}$,
and that $S_i$ contains at most one vertex from $\{ a_i,b_i,c_i\}$,
where the vertices $a_i$, $b_i$, and $c_i$ are as indicated in Figure \ref{fig2}.
This implies that, for $i\geq 2$, 
$$w_{(T_k',S_i)}(b_{i-1})\geq w_{(T_k',L_i)}(b_{i-1})=\frac{11}{32},$$
and that, for $i<k$,
$$w_{(T_k',S_i)}(a_{i+1})\geq w_{(T_k',L_i)}(a_{i+1})=\frac{23}{64}.$$
Now, let $i$ be such that $2\leq i\leq k-1$.
If $a_i\in S_i$, then
\begin{eqnarray*}
w_{(T_k',S\setminus \{ a_i\})}(a_i) 
& \geq & w_{(T_k',S_{i-1})}(a_i)+w_{(T_k',L_i)}(a_i)+w_{(T_k',S_{i+1})}(a_i)
\geq \frac{23}{64}+\frac{11}{16}+\frac{1}{2}\cdot \frac{11}{32}
>1,
\end{eqnarray*}
if $b_i\in S_i$, then
\begin{eqnarray*}
w_{(T_k',S\setminus \{ b_i\})}(b_i) 
& \geq & w_{(T_k',S_{i-1})}(b_i)+w_{(T_k',L_i)}(b_i)+w_{(T_k',S_{i+1})}(b_i)
\geq \frac{1}{2}\cdot \frac{23}{64}+\frac{23}{32}+\frac{11}{32}
>1,
\end{eqnarray*}
and,
if $c_i\in S_i$, then
\begin{eqnarray*}
w_{(T_k',S\setminus \{ c_i\})}(c_i) 
& \geq & w_{(T_k',S_{i-1})}(c_i)+w_{(T_k',L_i)}(c_i)+w_{(T_k',S_{i+1})}(c_i)
\geq \frac{1}{4}\cdot \frac{23}{64}+\frac{7}{8}+\frac{1}{8}\cdot\frac{11}{32}
>1.
\end{eqnarray*}
Since $S$ is exponentially independent, 
it follows that $S_i=L_i$ for every $i$ with $2\leq i\leq k-1$,
which implies the statement.
\end{proof}
Using 
in each $V_i$ the four vertices indicated by cycles 
in $V_1$ in Figure \ref{fig2},
and adding in every third $V_i$ the vertex indicated by the square 
in $V_1$ in Figure \ref{fig2}
yields an exponentially independent set in $T_k'$
of order $\frac{n(T_k')+O(1)}{3}$.

\section{Conclusion}

Our results raise numerous questions; we summarize those that seem most interesting to us. 
\begin{itemize}
\item {\it Do subcubic graphs have exponentially independent sets of linear order?}
(We believe that they do not.)
\item {\it Do trees of maximum degree at most $4$ have exponentially independent sets of linear order?}
(We believe that they do.)
\item {\it Do subcubic trees of order $n$ have exponentially independent sets of order $n/3+O(1)$?}
(We believe that they do.)
\item {\it What is the computational complexity of the exponential independence number 
within the class of subcubic trees?}
(We believe that the exponential independence number is NP-hard for subcubic trees.)
\end{itemize}
It is a trivial observation that 
every maximal independent set is dominating, 
which implies that the independence number of a graph 
is an upper bound on its domination number.
While it is easy to construct 
maximal exponentially independent sets 
that are not exponentially dominating, 
it might still be true that the exponential independence number  
of a graph is an upper bound on its exponential domination number.


\begin{thebibliography}{}
\bibitem{abcvy} M. Anderson, R.C. Brigham, J.R. Carrington, R.P. Vitray, J. Yellen, On exponential domination of $C_m\times C_n$, AKCE International Journal of Graphs and Combinatorics 6 (2009) 341-351.
\bibitem{aa1} A. Ayta\c{c}, B. Atay, On exponential domination of some graphs, Nonlinear Dynamics and Systems Theory 16 (2016) 12-19.
\bibitem{aa2} A. Ayta\c{c}, B. Atay Atakul, Exponential domination critical and stability in some graphs, International Journal of Foundations of Computer Science 30 (2019) 781-791.
\bibitem{bor1} S. Bessy, P. Ochem, D. Rautenbach, Bounds on the exponential domination number, Discrete Mathematics 340 (2017) 494-503.
\bibitem{bor2} S. Bessy, P. Ochem, D. Rautenbach, Exponential domination in subcubic graphs, The Electronic Journal of Combinatorics 23 (2016) $\#$P4.42.
\bibitem{ca} C. \c{C}ift\c{c}i, A. Ayta\c{c}, Exponential independence number of some graphs, International Journal of Foundations of Computer Science 29 (2018) 1151-1164.
\bibitem{ddems} P. Dankelmann, D. Day, D. Erwin, S. Mukwembi, H. Swart, Domination with exponential decay, Discrete Mathematics 309 (2009) 5877-5883.
\bibitem{hjr1} M.A. Henning, S. J\"{a}ger, D. Rautenbach, Relating Domination, Exponential Domination, and Porous Exponential Domination, Discrete Optimization 23 (2017) 81-92.
\bibitem{hjr2} M.A. Henning, S. J\"{a}ger, D. Rautenbach, Hereditary Equality of Domination and Exponential Domination, Discussiones Mathematicae Graph Theory 38 (2018) 275-285.
\bibitem{jr} S. J\"{a}ger, D. Rautenbach, Exponential Independence, Discrete Mathematics 340 (2017) 2650-2658.
\end{thebibliography}
\end{document}